\documentclass[a4paper]{article}
\usepackage{amsthm,amssymb,amsmath,enumerate,graphicx,epsf,bbold}

\newcommand{\COLORON}{0}
\newcommand{\NOTESON}{0}
\newcommand{\Debug}{0}
\usepackage[usenames,dvips]{color} 
\usepackage{amsthm,amssymb,amsmath,enumerate,graphicx,epsf,mathbbol,stmaryrd}
\usepackage[dvips, bookmarks, colorlinks=false, breaklinks=true]{hyperref}

\newcommand{\comment}[1]{}
\newcommand{\COMMENT}[1]{}

\definecolor{darkgray}{rgb}{0.3,0.3,0.3}
\newcommand{\defi}[1]{{\color{darkgray}\emph{#1}}}

\newcommand{\acknowledgements}{\section*{Acknowledgements}}

\comment{
	\begin{lemma}\label{}	
\end{lemma}
\begin{proof}

\end{proof}

\begin{theorem}\label{}
\end{theorem} 
\begin{proof} 	

\end{proof}

}

\newtheorem{proposition}{Proposition}[section]
\newtheorem{definition}[proposition]{Definition}
\newtheorem{theorem}[proposition]{Theorem}
\newtheorem{corollary}[proposition]{Corollary}
\newtheorem{lemma}[proposition]{Lemma}

\newtheorem{conjecture}{{Conjecture}}[section]

\newtheorem{problem}[conjecture]{{Problem}}

\newtheorem{examp}[proposition]{Example}

\newcommand{\FIG}{0}

\ifnum \NOTESON = 1 \newcommand{\note}[1]{ 

	\ 

	{\color{blue} \hspace*{-60pt} NOTE: \color{Turquoise}{\small  \tt \begin{minipage}[c]{1.1\textwidth}  #1 \end{minipage} \ignorespacesafterend }} 
	
	\ 
	
	}
\else \newcommand{\note}[1]{} \fi

\newcommand{\afsubm}[1]{ \ifnum \Debug = 1 {\mymargin{#1}}
\fi} 

\ifnum \Debug = 1 
\else  \fi

\ifnum \FIG = 1 
\else  \fi

\ifnum \FIG = 1 
\else  \fi

\ifnum \Debug = 1 \usepackage[notref,notcite]{showkeys}
\fi

\ifnum \COLORON = 0 \renewcommand{\color}[1]{}
\fi

\newcommand{\R}{\ensuremath{\mathbb R}}

\newcommand{\sm}{\backslash}

\newcommand{\sgl}[1]{\ensuremath{\{#1\}}}

\newcommand{\g}{\ensuremath{G\ }}
\newcommand{\G}{\ensuremath{G}}

\newcommand{\knl}{Kirchhoff's node law}
\newcommand{\kcl}{Kirchhoff's cycle law}

\newcommand{\are}{\vec{e}}
\newcommand{\arE}{\vec{E}}

\newcommand{\Tr}[1]{Theorem~\ref{#1}}
\newcommand{\Sr}[1]{Section~\ref{#1}}
\newcommand{\Prr}[1]{Pro\-position~\ref{#1}}

\newcommand{\Cr}[1]{Corollary~\ref{#1}}

\renewcommand{\iff}{if and only if}
\newcommand{\fe}{for every}
\newcommand{\Fe}{For every}

\newcommand{\st}{such that}

\newcommand{\wrt}{with respect to}

\newcommand{\rw}{random walk}

\newcommand{\labtequ}[2]{ \begin{equation} \label{#1} 	\begin{minipage}[c]{0.9\textwidth}  #2 \end{minipage} \ignorespacesafterend \end{equation} } 

\newcommand{\labtequc}[2]{ \begin{equation} \label{#1} 	\text{   #2 } \end{equation} }

\newcommand{\mymargin}[1]{
  \marginpar{%
    \begin{minipage}{\marginparwidth}\small%
      \begin{flushleft}%
        {\color{blue}#1}%
      \end{flushleft}%
   \end{minipage}%
  }%
}%

\newcommand{\mySection}[2]{}

\newcommand{\hit}[2]{H_{#1}(#2)}

\newcommand{\hv}{\hat{v}}
\newcommand{\hi}{\hat{i}}

\title{A Tractable Variant of Cover Time}

\author{Agelos Georgakopoulos\thanks{Supported by FWF Grant P-24028-N18.}
}

\begin{document}
\maketitle

\begin{abstract}
We introduce a variant of the cover time of a graph, called cover cost, in which the cost of a step is proportional to the number of yet uncovered vertices. It turns out that cover cost is more tractable than cover time; we provide an $O(n^4)$ algorithm for its computation, as well as some explicit formulae. The two values are not very far from each other, and so cover cost might be a useful tool in the study of cover time.  
\end{abstract}

\section{Introduction}
The cover time $CT_r(G)$ of a graph \g from a vertex $r$ is the expected number of steps it takes for random walk on \g starting at $r$ to visit all vertices. The  cover time of \g is defined as $CT(G):= \max_{r\in V(G)} CT_r(G)$. It has been extensively studied in various contexts;
applications include the construction of universal traversal sequences \cite{AKLLR, BrUni},
testing graph connectivity \cite{AKLLR, KaRaRan}, and protocol testing \cite{MiPaRan}. It has been studied by physicists interested in the fractal structure of the uncovered set of a finite grid; see \cite{DPRZ} for references and for an interesting relation between
the cover time of a finite grid and Brownian motion on Riemannian manifolds. 

Many bounds on cover time for specific classes of graphs have been obtained, see e.g.\ \cite{BW2,FeiTig,edgecov} and references therein. However, in general it is very hard to obtain exact formulae for cover time, or compute it algorithmically. A question of \cite{AldousFill} that remained open for many years despite several efforts \cite{MatCov,KKLV}, and was recently resolved in the affirmative using heavy probabilistic machinery \cite{DLP}, is whether there is a deterministic algorithm which approximates $CT(G)$ up to a constant factor in polynomial time.

\medskip
The following situation suggests a similar concept, which we will call the \defi{cover cost} of \G. Suppose that a truck starts at $r$ loaded with some goods to be equally distributed to the other vertices of the graph and performs random walk, leaving the goods corresponding to each vertex  the first time it gets there and carrying all remaining goods along. Rather than the expected number of steps it will take to finish this tour (which equals $CT_r(G)$), the truck driver is more interested in the expected total amount of goods that will have to be carried.

\begin{definition}
Consider a cover tour of random walk on a graph $G=(V,E)$  starting at vertex $r$. Define the cost of step $i$ to be $1 - \frac{k}{n}$ where $k$ is the (random) number of vertices visited so far excluding $r$, and $n:=|V|-1$. Define the \defi{cover cost} $cc_r(G)$ to be the expected total cost of all steps in the cover tour.
\end{definition}

As an example, let $S_n$ be a star with center $r$ and $n$ leaves. The problem of determining $CT_r(S_n)$ is the well known coupon collection problem, and one has $CT_r(S_n) \approx 2n\ln n$ \cite{AldousFill}. It is easy to check that $cc_r(S_n) \approx 2n$: having visited the first $k$ leaves, it will take an expected $\frac{2n}{n-k}$ steps to reach the next unvisited vertex, and almost each such step will cost $\frac{n-k}{n}$; the step from the newly reached leaf back to $r$ will be cheaper by $1/n$, and so $cc_r(S_n) = 2n-1$.
\medskip


For reasons that will become clear soon, we also define the (uppercase) \defi{Cover Cost} by $CC_r(G):= n cc_r(G)$ for $n=|V|-1$. Note that $cc_r(G)<CT_r(G)<CC_r(G)$. 
There is a further intuitive way of defining Cover Cost. Suppose now that instead of a  truck driver in the above game we have an electrician trying to visit all vertices of a graph to do some repair. Define his cost $ECC_r(G)$ to be the sum of the expected waiting times of his clients\footnote{This parameter was proposed by Peter Winkler (private communication).}. An easy double-counting argument implies that
$$ECC_r(G) = CC_r(G).$$
Note that, by linearity of expectation, $ECC_r(G) = \sum_{x\in V} \hit{r}{x}$, where $\hit{r}{x}$ is the \defi{hitting time} from $r$ to $x$, i.e.\ the expected time for \rw\ starting at $r$ to reach $x$. Thus
\labtequc{defh}{$CC_r(G)= \sum_{x\in V} \hit{r}{x}.$}
This formula yields a (deterministic) algorithm with running time $O(n^4)$ computing cover cost precisely. We describe this algorithm in \Sr{algo}. 

\medskip
In \Sr{secMain} we prove
\begin{theorem} \label{TCCpp}
For every graph $G=(V,E)$ and every $r\in V$, we have
$$CC_{r}(G)=  \sum_{x,y\in V} \frac{p_r(x<y)}{p_{xy}}$$
where $p_{xy}$ is the probability  that random walk from $x$ will visit $y$ before returning to $x$, and $p_r(x<y)$ is  the probability  that random walk from $r$ will visit $x$ before returning to $y$.
\end{theorem}

In the case where \g is a tree, \Tr{TCCpp} yields a simpler expression for cover cost involving no probabilistic parameters (see \Cr{tree}), and it implies a surprising connection to the Wiener index proved in \cite{CTree}; see \Sr{secMain}.

\medskip
In the case where \g is the path $P_n$ of $n$ edges we show that cover cost has the same order of magnitude $O(n^2)$ as cover time. However, the two parameters behave differently when $n$ is fixed and we vary the starting point of our cover tour on $P_n$. Both $cc_r(P_n)$ and $CT_r(P_n)$ are minimised when $r$ is an endpoint and maximised when $r$ is a midpoint, but the difference between the two extremes behaves very differently: for $CT_r(P_n)$ this difference is also $O(n^2)$, about $CT_r(P_n)/4$, while for $cc_r(P_n)$ the difference is $O(n)$.

\begin{theorem}\label{main}
For a path $P_n$ with an even number of edges $n$, cover cost $cc_r(P_n)$ is minimised when $r$ is an endpoint and maximised when $r$ is the midpoint. Moreover, $cc_{\frac{n}{2}}(P_n)- cc_0(P_n)= n/4$.
\end{theorem}

We prove this in \Sr{pths}; in fact, we will give an exact formula for $cc_r(P_n)$ for every starting point $r$.

\medskip
In \Sr{prob} we propose some problems on cover cost and its relation to cover time.

\bigskip
Throughout this paper, a {\em random walk} begins at some vertex $r$ of a finite graph \G, and when at vertex $x$, it chooses one of the neighbours of $x$ at random according to the uniform distribution, and moves to that neighbour.

\section{Paths}\label{pths}

In this section we obtain an explicit formula for the cover cost of a path, from which \Tr{main} immediately follows:

\begin{proposition}\label{pr}
Let $P_n$ be the path on $n+1$ vertices, indexed by $0,1,\ldots n$. Then $cc_r(P_n)=\frac{(n+1)(2n+1)}{6} + \frac{r(n-r)}{n}$.
\end{proposition} 
\begin{proof} 	
We are going to consider the sum of the hitting times $\hit{r}{k}$ from $r$ and apply \eqref{defh}. For $r=0$ we have the well known formula $\hit{0}{k}=k^2$. For $k>r>0$, we have $ \hit{0}{k} = \hit{0}{r} + \hit{r}{k}$. Combining these two formulas we get $\hit{r}{k}= k^2-r^2$ for $k>r$, and similarly we get $\hit{r}{k}= (n-k)^2-(n-r)^2 = (r-k)(2n-(k+r))$ for $k<r$.

Thus we have
\begin{eqnarray*}
\sum_{k} \hit{r}{k} =& \sum_{k>r} (k^2-r^2) + & \sum_{k<r} (r-k)(2n-(k+r))\\
=&\sum_{k>r} (k^2-r^2) + & \sum_{k<r} (k-r)(k+r) + \sum_{k<r} 2n(r-k)\\
=&\sum_{k\neq r} (k^2-r^2) + & \sum_{k<r} 2n(r-k),\\
\end{eqnarray*}
and using the formulae for the sum of the first $r$ squares and the first $r$ natural numbers we can rewrite this as 
\begin{eqnarray*}
\sum_{k} \hit{r}{k} =&\frac{n(n+1)(2n+1)}{6}- r^2 & -nr^2 + 2n\frac{(r+1)r}{2}\\
=&\frac{n(n+1)(2n+1)}{6} + & r(n-r).
\end{eqnarray*}

Plugging this into \eqref{defh} we obtain the desired formula.
\end{proof}

\section{General Formulae for Cover Cost} \label{secMain}

Let $G=(V,E)$ be a graph on $n+1$ vertices, and fix $r\in V$. We can express $CC_{r}$ as the sum of the contribution of each $x\in V$ to $CC_{r}$ as follows. Start a \rw\ particle of `charge' 1 at $r$, and each time the particle visits a vertex in $V\sm \sgl{r}$ for the first time, reduce its charge by $1/n$, letting it continue its \rw\ with the remaining charge as long as this charge is non-zero. Note that the particle is stopped upon completing a cover tour. \Fe\ $x\in V$, let $D(x)$ denote the expected total amount of charge departing from $x$ in this \rw. By the definitions, we have 
\labtequc{ccD}{$\frac{CC_{r}(G)}{n}= cc_{r}(G)= \sum_{x\in V} D(x)$.}
Next, we are going to split each $D(x)$ into contributions of each $y\in V$ as follows. Let $V_{x}^{\dagger y}$ denote the expected number of times that random walk from $x$ will visit $x$ before visiting $y$ for the first time; we also count the starting step as a visit to $x$, and so $V_{x}^{\dagger y}>1$ \fe\ $y\neq x$. Let $p_r(x<y)$ denote the probability for random walk from $r$ to visit $x$ before $y$.

We claim that 
\labtequc{Vrx}{$D(x)= \frac1{n}\sum_{y\in V \sm \{x\} } p_r(x<y)V_{x}^{\dagger y}$.}
To see this, think of the initial charge of the particle as having been divided into $n$ `quarks' of charge $1/n$ before beginning the tour, each quark labelled by a distinct vertex at which it is meant to be left. This allows as to write $D(x)$ as the sum of the expected contributions of each quark to $D$. Linearity of expectation now implies the above formula.

Using the formula for the expected number of repetitions of a Bernoulli trial until the first success, we see that $V_{x}^{\dagger y}=1/p_{xy}$ where $p_{xy}$ is the probability  that random walk from $x$ will visit $y$ before returning to $x$. Combining this to the above formulas we get
\labtequc{ccpp}{$CC_{r}(G)= \sum_{x,y\in V} \frac{p_r(x<y)}{p_{xy}}$.}
This proves \Tr{TCCpp}. 

\bigskip

In the case where $G$ is a tree, the probabilistic parameters in \eqref{ccpp} can be replaced by explicit graph-theoretic ones, yielding a pleasant formula for the cover cost. Given vertices $r,x,y$ on a tree, let $x \wedge_r y$ denote the \defi{confluent} of $x,y$ \wrt\ $r$, i.e.\ the vertex of minimal distance from $r$ separating $x$ from $y$. Recall that {$d(x)$} is the degree of $x$. We have
\begin{corollary}[\cite{CTree}]\label{tree}
Let $T$ be a tree and $r\in V(T)$. Then 
$$CC_{r}(T)=  \sum_{x,y\in V(T)} d(x \wedge_r y, y) d(x).$$
\end{corollary}

This is proved in \cite{CTree}, where it is further used to obtain an interesting connection between $CC_r(T)$ and the Wiener index $W(T)$ of $T$: it is proved that, \fe\ tree $T$,
$$\sum_{v\in V(T)} \left( H_{rv} + d(r,v) \right) = CC_r(T)+ \sum_{v\in V(T)} d(r,v) = 2W(T) := \sum_{x,y\in V(T)} d(x,y).$$

\comment{
\begin{proof} 

We claim that $p_{xy}=\frac1{d(x) d(x,y)}$ on a tree, and so our assertion follows easily from \eqref{ccpp}. To begin with, notice that $p_{xy}$ only depends on the degree of $x$ and the length $d(x,y)$ of the $x$-$y$~path, since any ramifications on that path do not affect $p_{xy}$. Moreover, \Prr{prh} easily implies
\labtequ{a}{Let $P_n$ be the path on $n+1$ vertices, indexed by $0,1,\ldots n$. Then the probability $p_r(0<n)$ for random walk from $r$ to visit vertex $0$ before vertex $n$ is $\frac{n-r}{n}$. Moreover, $p_{0n}=1/n$.}
Thus $p_{xy} = 1/d(x) d(x,y)$ in our case. Substituting in \eqref{ccpp} we obtain
$$CC_r(T)=  \sum_{x,y\in V} {p_r(x<y)}{d(x) d(x,y)}$$
where $d(x)$ denotes the degree of $x$ and $d(x,y)$ the usual edge-counting distance.

The values $p_r(x<y)$ in this formula can also be expressed in terms of distances by the same arguments: we have $p_r(x<y)= d(x \wedge_r y, y)/ d(x,y)$. Substituting these values above we obtain the desired  formula.
\end{proof} 
}


\section{Algorithms} \label{algo}

Formula \eqref{defh} allows for an efficient computation of the exact value of cover time: fixing $x\in V$, we can write
\labtequ{H}{$H_y(x) = 1+ \frac{\sum_{\{z\mid yz\in E\}} H_z(x)}{d(y)}$}
\fe\ $y\neq x\in V$, since the first step of \rw\ from $y$ takes us to some neighbour of $y$. This, and the fact that $H_x(x)=0$, yields a system of $n$ linear equations with $n$ unknowns, which can be solved  in time $O(n^3)$, e.g.\ by Gaussian elimination. In order to obtain $CC_r(G)$ it suffices, by \eqref{defh}, to solve $n$ such systems, one for each $x\neq r$, and add up the values $H_r(x)$. Thus $CC_r(G)$ can be computed in time $O(n^4)$; in fact, with almost no additional effort we compute simultaneously the Cover Cost \fe\ starting vertex $r$.\footnote{I was made aware of this algorithm by Erol Pekoz.}

A further algorithm for the computation of $CC_r(G)$ can be derived from the proof of \Tr{TCCpp} as follows. For  the quantity $D(r)$ of \eqref{ccD} we have, by \eqref{Vrx} and the discussion following it, $D(r)= \sum_{y\in V}  1/p_{ry}$. Now each value $p_{ry}$ can be computed by solving a linear system of $n$ equations similar to \eqref{H}: we have
$p_{ry}= \sum_{\{z\mid rz\in E\}} p_z(y<r)$ and
$p_z(y<r)= \sum_{\{z'\mid zz'\in E\}} p_{z'}(y<r)$ \fe\ $z\neq r,y$, while $p_y(y<r)=1$.
Once $D(r)$ has been computed, all other values $D(x), x\in V\sm \sgl{r}$ can be obtained simultaneously by solving a single linear system of $n$ equations: by the definition of $D(x)$ we have
$$D(x)= \sum_{\{z\mid xz\in E\}} \frac{D(z)}{d(z)}.$$
Thus this method yields a further $O(n^4)$ algorithm for the computation of $CC_r(G)$.


\section{Further Problems} \label{prob}

The examples in the introduction show that cover cost might, depending on the graph, have the same order of magnitude as cover time or be quite smaller. How does cover cost behave in the extremal cases of `fast graphs', i.e.\ when $CT= O(n\ln n)$, and `slow graphs', i.e.\ when $CT(G)=  O(n^3)$? (It is known that $n\ln n \lesssim CT(G) \lesssim 4n^3/27$ for every graph $G$ on $n$ vertices \cite{FeiTig}).

\begin{problem}
Is it true that \fe\ $M\in \R^+$ there is $c(M)$ \st\ for every graph \G\ and $r\in V(G)$, if $CT_r(G)<M n\ln n$ then $cc_r(G)<c(M) n$, where $n=|V(G)|-1$?\footnote{This problem arrised after a discussion with Itai Benjamini.}
\end{problem}

\begin{problem}
Is it true that \fe\ $M\in \R^+$ there is $c(M)$ \st\ for every graph \G\ and $r\in V(G)$, if $CT_r(G)> M n^3$ then $cc_r(G)> c(M) n^3$?
\end{problem}

The  extremal graphs for cover time are not known, although a lot of work has been done and graphs that are close to being extremal are known \cite{FeiCol}. It would be interesting to find the extremal graphs for cover cost:
\begin{problem}
Which rooted graph on $n$ vertices minimises $cc_r(G)$? Which maximises it? 
\end{problem}

\medskip
It would be very interesting to obtain bounds on $cc_r(G)/CT_r(G)$, for this would allow us to use cover cost in order to obtain bounds on cover time. The following two problems are motivated by this.

\begin{conjecture}
The path on $n$ vertices rooted at an endpoint maximises $cc_r(G)/CT_r(G)$ over all rooted graphs \g on $n$ vertices. 
\end{conjecture}

\begin{problem}
Is there a graph \g for which $cc_r(G)/CT_r(G)< 1/H_{(|V(G)|-1)}$ where $H_n\approx \ln n$ is the $n$th harmonic number? Which graph on $n$ vertices minimises $cc_r(G)/CT_r(G)$?
\end{problem}

\acknowledgements{I am very grateful to Itai Benjamini, Erol Pekoz, and Peter Winkler for their aforementioned suggestions and other valuable discussions.}

\bibliographystyle{amsalpha}
\bibliography{../collective}
\end{document}